\documentclass[english,11pt]{article}

\usepackage{amssymb,amsmath,amsfonts,amsthm,rotating,setspace,graphicx,float}
\usepackage{amssymb,amsmath,amsthm,rotating,setspace}
\usepackage{enumerate}

\usepackage[title]{appendix}

\usepackage{graphicx}

\usepackage{latexsym}

\usepackage{epsfig}

\usepackage{color}

\newtheorem{theorem}{Theorem}[section]
\newtheorem{lemma}[theorem]{Lemma}

\newtheorem{cor}[theorem]{Corollary}

\theoremstyle{definition}

\theoremstyle{definition}

\numberwithin{equation}{section}

\def\eq#1{(\ref{#1})}

\def\ve{\varepsilon}

\def\N{\mathbb{N}}

\def\R{\mathbb{R}}

\begin{document}

\title{Classification of radial solutions for elliptic systems  driven by the $k$-Hessian operator}


\author{Marius Ghergu\footnote{School of Mathematics and Statistics,
    University College Dublin, Belfield, Dublin 4, Ireland; {\tt
      marius.ghergu@ucd.ie}}$\,$ \footnote{Institute of Mathematics Simion Stoilow of the Romanian Academy, 21 Calea Grivitei Street,
010702 Bucharest, Romania} 
}




\maketitle

\begin{abstract}
We are concerned with non-constant positive radial solutions of the system
$$
\left\{
\begin{aligned}
S_k(D^2 u)&=|\nabla u|^{m} v^{p}&&\quad\mbox{ in }\Omega,\\
S_k(D^2 v)&=|\nabla u|^{q} v^{s} &&\quad\mbox{ in }\Omega,
\end{aligned}
\right.
$$
where $S_k(D^2u)$ is the $k$-Hessian operator of $u\in C^2(\Omega)$ ($1\leq k\leq N$) and $\Omega\subset\R^N$ $(N\geq 2)$ is either a ball or the whole space. The exponents satisfy $q>0$, $m,s\geq 0$, $p\geq s\geq 0$ and $(k-m)(k-s)\neq pq$.
In the case where $\Omega$ is a ball, we classify all the positive radial solutions according to their behavior at the boundary. Further, we consider the case $\Omega=\R^N$ and find that the above system admits non-constant positive radial solutions if and only if $0\leq m<k$ and $
pq < (k-m)(k-s)$. Using arguments from three component cooperative and irreducible dynamical systems we deduce the behavior at infinity of such solutions.

\end{abstract}


\medskip



\section{Introduction}

In this paper we study positive non-constant radially symmetric solutions of the system
\begin{equation}\label{hess}
\left\{
\begin{aligned}
S_k(D^2 u)&=|\nabla u|^{m} v^{p}&&\quad\mbox{ in }\Omega,\\
S_k(D^2 v)&=|\nabla u|^{q} v^{s} &&\quad\mbox{ in }\Omega,
\end{aligned}
\right.
\end{equation}
where $\Omega$ is either an open ball $B_R\subset \R^N$ $(N\geq 2)$, centred at the origin and having radius $R>0$, or $\Omega=\R^N$. The exponents $m,p,q,s$ are assumed  to satisfy
$$
q>0,\quad  m,s\geq 0, \quad p\geq s\geq 0
$$
and
\begin{equation}\label{deltat}
\delta:=(k-m)(k-s)-pq\neq 0.
\end{equation}

Throughout this paper, $S_k(D^2u)$ denotes the $k$-Hessian operator of $u\in C^2(\Omega)$, $1\leq k\leq N$, defined as follows. 
Let $\Lambda=(\lambda_1, \lambda_2, \dots, \lambda_N)$ be the eigenvalues of the Hessian matrix $D^2u$. Then,
$$
S_k(D^2 u)=P_k(\Lambda)=\sum_{1\leq i_1<i_2<\dots<i_k\leq N}\lambda_{i_1}\lambda_{i_2} \dots \lambda_{i_k},
$$
where $P_k(\Lambda)$ is the $k$-th elementary symmetric polynomial in the eigenvalues $\Lambda$. We point out that $\{S_k\}_{1\leq k\leq N}$ is a family of operators which contains the Laplace operator $(k = 1)$ and the Monge-Amp\`ere operator $(k =N)$. 

For $2\leq k\leq N$ the operators $S_k$ are fully nonlinear. Further, $S_k$ are not elliptic in general, unless they are restricted to the class
$$
\Gamma_k=\{u\in C^2(\Omega):S_i(D^2 u)\geq 0\mbox{ in }\Omega\mbox{ for all }1\leq i\leq k\}.
$$

In this paper we study {\it non-constant positive radial solutions} of \eqref{hess}, that is, solutions $(u,v)$ which fulfill:
\begin{itemize}
\item $u, v\in \Gamma_k$ are positive and radially symmetric;
\item $u$ and $v$ are not constant in any neighbourhood of the origin;
\item $u$ and $v$ satisfy \eqref{hess}.
\end{itemize}
If $\Omega=\R^N$, such solutions of \eqref{hess} will be called {\it global radial solutions}.

Throughout this paper, we identify radial solutions
$(u,v)$ with their one variable representant, that is, $u(x)=u(r)$,
$v(x)=v(r)$, $r=|x|$. It is now a standard argument (see, e.g., \cite{JB2010}) to check that any positive radial solution $(u,v)$ of \eqref{hess} in $B_R$ satisfies
\begin{equation}\label{main0}
\left\{
\begin{aligned}
&\binom{n-1}{k-1} r^{1-N}\Big[ r^{N-k}|u'|^{k-1}u' \Big]'=|u'(r)|^{m}v^{p}(r) &&\quad\mbox{ for all } 0<r<R,\\
&\binom{n-1}{k-1} r^{1-N}\Big[ r^{N-k} |v'|^{k-1}v'\Big]'=|u'(r)|^{q} v^{s}(r) &&\quad\mbox{ for all } 0<r<R,\\
&u'(0)=v'(0)=0, u(r)>0, v(r)>0 &&\quad\mbox{ for all } 0<r<R,
\end{aligned}
\right.
\end{equation}
where $\binom{n-1}{k-1}$ stands for the binomial coefficient for the integers $n-1\geq k-1$. A scaling argument yields easily that \eqref{main0} is equivalent to
\begin{equation}\label{main}
\left\{
\begin{aligned}
& r^{1-N}\Big[ r^{N-k}|u'|^{k-1}u' \Big]'=|u'(r)|^{m}v^{p}(r) &&\quad\mbox{ for all } 0<r<R,\\
& r^{1-N}\Big[ r^{N-k} |v'|^{k-1}v'\Big]'=|u'(r)|^{q} v^{s}(r) &&\quad\mbox{ for all } 0<r<R,\\
&u'(0)=v'(0)=0, u(r)>0, v(r)>0 &&\quad\mbox{ for all } 0<r<R.
\end{aligned}
\right.
\end{equation}

Partial differential equations related to the $k$-Hessian operator have been widely investigated in the last four decades. The results in Caffarelli, Nirenberg and Spruck \cite{CNS1985} (see also Ivochina \cite{I1980}) have opened up new mathematical methods in this direction. Ji and Bao \cite{JB2010} obtained Keller-Osserman type conditions for the existence of a solution to $S_k(D^2 u)=f(u)$ in the entire space $\R^N$.

The study of the system \eqref{hess} is motivated by the semilinear case
\begin{equation}\label{diaz1}
\left\{
\begin{aligned}
\Delta u &=v&&\quad\mbox{ in }\Omega,\\
\Delta v &=|\nabla u|^2&&\quad\mbox{ in }\Omega,
\end{aligned}
\right.
\end{equation}
discussed in \cite{DLS2005} as a steady state model of a viscous, heat conducting fluid. The time dependent version of \eqref{diaz1}, namely,
\begin{equation}\label{diaz2}
\left\{
\begin{aligned}
u_t-\Delta u &=\theta&&\quad\mbox{ in }\Omega\times (0,T),\\
\theta_t-\Delta \theta &=|\nabla u|^2&&\quad\mbox{ in }\Omega\times (0,T),
\end{aligned}
\right.
\end{equation}
is investigated in \cite{DRS2007} and \cite{DRS2008}. In the above coupled equations, $u$ stands for the speed and $\theta$ stands for the temeprature 
of a unidirectional flow,
independent of distance in the flow direction. Note that the steady states of \eqref{diaz2} corresponds after the change $\theta=-v$ to solutions of system \eqref{diaz1}. Further extensions of \eqref{diaz1} to the case of general nonlinearities appear in Singh \cite{S2015}, Filippucci and
Vinti \cite{FV2017}. A recent work of Ghergu, Giacomoni and Singh  \cite{GGS2019} investigates radial solutions of the quasilinear system
$$
\left\{
\begin{aligned}
\Delta_\sigma u&=|\nabla u|^{m} v^{p}&&\quad\mbox{ in }\Omega,\\
\Delta_\sigma v&=|\nabla u|^{q} v^{s} &&\quad\mbox{ in }\Omega,
\end{aligned}
\right.
$$
where $\Delta_\sigma u={\rm div}(|\nabla u|^{\sigma-2}\nabla u)$, $\sigma>1$.

We should point out that the system \eqref{hess} and its radially symmetric counterpart \eqref{main0} is not a singular system (as for instance in \cite{G2009, GR2008}). The difficulty in the study of \eqref{hess} lies in the presence of the gradient terms $|\nabla u|^m$ and $|\nabla
u|^q$  in the right-hand side of \eqref{hess} which, as we shall see, leads to a rich structure of the solution set.   In the following, for a function $f:(0,R)\to
\R$ we denote $f(R^-)=\lim_{r\nearrow R} f(r)$, provided such a
limit exists. Also, $C,c,c_1,c_2...$ stand for positive constants whose values may chang on each occurence.

Our first result is concerned with the case where $\Omega$ is a ball.

\begin{theorem}\label{thm1}
Assume $\Omega=B_R$, $q>0$, $m,s\geq 0$, $p\geq s\geq 0$ and $\delta\neq 0$. Then:
\begin{enumerate}
\item [(i)] There are no positive radial solutions $(u, v)$ of \eqref{hess} with $u(R^{-})=\infty$ and $v(R^{-})<\infty$.
\item [(ii)] All positive radial solutions of \eq{hess} are bounded if and only if
$$
k>m\quad\mbox{ and}\quad 
pq < (k-m)(k-s).
$$
\item [(iii)] There are positive radial solutions $(u, v)$ of \eq{hess} with $u(R^{-})<\infty$ and $v(R^{-})=\infty$
if and only if
$$
k>m\quad\mbox{ and}\quad 
pq> p(k+1)+(k-m+1)(k-s).
$$
\item [(iv)] There are positive radial solutions $(u, v)$ of \eq{hess} with $u(R^{-})= v(R^{-})=\infty$ if and only if
$$
k>m\quad\mbox{ and}\quad 
(k-m)(k-s)<pq\leq  p(k+1)+(k-m+1)(k-s).
$$
\end{enumerate}

\end{theorem}

As an immediate consequence of the above result we obtain optimal conditions for the existence of boundary blow-up solutions for equations and system. In such a setting, the Keller-Osserman condition plays a crucial role (see, e.g. \cite{DGGW2012, GR2008book, MP2019}). 

Let us first consider the boundary blow-up system:
\begin{equation}\label{hessblow1}
\left\{
\begin{aligned}
S_k(D^2 u)&=|\nabla u|^{m} v^{p}&&\quad\mbox{ in }B_R,\\
S_k(D^2 v)&=|\nabla u|^{q} v^{s} &&\quad\mbox{ in }B_R,\\
u=v&=\infty &&\quad\mbox{ on }\partial B_R,
\end{aligned}
\right.
\end{equation}
where the boundary condition in \eqref{hessblow1} is understood in the following sense
$$
u(x)\to \infty, \ v(x)\to \infty\quad\mbox{ as }{\rm dist}(x, \partial B_R)\to 0.
$$
From Theorem \ref{thm1} (iv) we find:
\begin{cor}\label{ccblow1}
Assume $q>0$, $m,s\geq 0$, $p\geq s\geq 0$ and $\delta\neq 0$. 
Then, \eqref{hessblow1} admits positive radial solutions if and only if
$$
k>m\quad\mbox{ and}\quad 
(k-m)(k-s)<pq\leq  p(k+1)+(k-m+1)(k-s).
$$
\end{cor}
Corollary \ref{ccblow1} provides optimal conditions for the existence of radial solutions of the   problem:
\begin{equation}\label{hessblow2}
\left\{
\begin{aligned}
S_k(D^2 u)&=u^p |\nabla u|^{q}&&\quad\mbox{ in }B_R,\\
u&=\infty &&\quad\mbox{ on }\partial B_R.
\end{aligned}
\right.
\end{equation}
As before, the boundary condition in \eqref{hessblow2} means $u(x)\to \infty$ as ${\rm dist}(x, \partial B_R)\to 0$. Letting $m=q$ and $p=s$ in Corollary \ref{ccblow1} we find:

\begin{cor}\label{ccblow2}
Assume $p\geq 0$, $q>0$ and $(k-p)(k-q)\neq pq$. 
Then, \eqref{hessblow2} admits positive radial solutions if and only if
$$
k>q\quad\mbox{ and}\quad 
pq>(k-p)(k-q).
$$
\end{cor}
We point out that the case $q=0$ in \eqref{hessblow2} is discussed in \cite{ZF2020}.

In the following we shall be concerned the case where $\Omega$ coincides with the whole space $\R^N$. Directly from Theorem \ref{thm1} one has:

\begin{cor}\label{cc1}
Assume $\Omega=\R^N$, $q>0$, $m,s\geq 0$, $p\geq s\geq 0$ and $\delta\neq 0$. 
Then, \eqref{hess} admits non-constant global positive radial solutions if and only if
\begin{equation}\label{nonct}
0\leq m<k \quad\mbox{ and}\quad 
pq < (k-m)(k-s).
\end{equation}
\end{cor}
The second condition in \eqref{nonct} reads $\delta>0$.

We next study the exact behavior at infinity of global positive radial solutions of \eqref{hess}. Note that the system \eqref{hess} is equivalent to \eqref{main0} in the radial setting. For the ease of our exposition, we shall discuss the behavior at infinity of solutions to the equivalent system \eqref{main}. In this direction we obtain the following result:

\begin{theorem}\label{thm3}
Assume $\Omega=\R^N$, $0\leq m<k$ and $\delta>0$. Then, any
non-constant global positive solution $(u, v)$ of \eq{main}
satisfies
\begin{equation}\label{abconst}
\lim_{|x|\rightarrow \infty} \frac{u(x)}{|x|^{{1+\frac{k(k-s+2p)}{\delta}}}}= A \quad \mbox{ and }\quad
\lim_{|x|\rightarrow \infty} \frac{v(x)}{|x|^{\frac{k(2k-2m+q)}{\delta}} }= B,
\end{equation}
where $A= A(N, m,p, q, s)>0$ and $B= B(N, m,p, q, s) >0$ are given by \eq{th2i} and \eq{th2j}.
\end{theorem}
Theorem \ref{thm3} above roughly says that any global radial solution $(u,v)$ of \eqref{main} stabilizes to 
$$
(U,V)=\Big(A|x|^{{1+\frac{k(k-s+2p)}{\delta}}},  B|x|^{\frac{k(2k-2m+q)}{\delta}}\Big)
$$ 
which is in fact a singular solution of \eqref{main}. 
In obtaining the exact behavior \eqref{abconst} we employ some results from three-component irreducible dynamical systems  from Hirsch \cite{H1985}. We recall these results in the first part of Section 3.

We point out that the requirement $\delta>0$ in \eqref{deltat} is a classical condition on superlinearity
of the system as it appears for instance in \cite{BVG2010}. Also, the value of the limits $A$ and $B$ in \eqref{abconst}
depend decreasingly on the space dimension $N\geq 2$. One may see this fact from their expressions in \eq{th2i} and \eq{th2j}.

In our next result we show that given any pair $(a,b)\in (0,\infty)\times (0,\infty)$, there exists a
unique positive global radial solutions of \eqref{hess} that emanates from $(a,b)$.

\begin{theorem}\label{thm4}
Assume $\Omega=\R^N$, $0\leq m<k$, $\delta>0$ and $1\leq k<N/2$. Then, for any $a> 0$, $b> 0$ there exists a unique non-constant
global positive radial solution of \eq{hess} such that $u(0)= a$ and $v(0)= b$.
\end{theorem}
Finally, let us discuss the single equation 
\begin{equation}\label{eqqh}
S_k(D^2 u)=u^p|\nabla u|^q\quad\mbox{ in }\R^N, N\geq 2,
\end{equation}
which the prototype of our system \eqref{hess}. 
The case $q=0$ was discussed in \cite{BJL2012} and \cite{JB2010}. By taking $m=q$ and $p=s$ in Corollary \ref{cc1} and Theorems \ref{thm3}-\ref{thm4} above we obtain:
\begin{cor}\label{coro}
Assume $\Omega=\R^N$, $1\leq k\leq N$, $p,q>0$ and $k\ne p+q$. 
Then \eqref{eqqh} has a non-constant positive radial solution if and only if $k>p+q$ and in this case, any non-constant positive radial solution $u$ of \eqref{eqqh} satisfies
$$
\lim_{|x|\to \infty}\frac{u(x)}{|x|^{\frac{2k-q}{k-p-q}}}=C(N,p,q)>0.
$$
If, in addition, $1\leq k<N/2$, then from any $a>0$ there exists a unique non-constant positive radial solution $u$ of \eqref{eqqh} such that $u(0)=a$.
\end{cor}

\section{Proof of Theorem \ref{thm1}}

Let us argue first that if $k\leq m$ then \eqref{main} has no solutions in $[0, R)$.

From \eqref{main} we have that the mappings $r\longmapsto r^{\alpha}|u'|^{\beta-1}u'$ and $r\longmapsto r^{\alpha}|v'|^{\beta-1}v'$ are increasing and since $u'(0)=v'(0)=0$, it follows that $u', v'\geq 0$ on $(0, R)$. In fact, integrating over $[r_0,r]$ in the first equation of \eqref{main} we find
$$
r^{N-k} (u'(r))^k =r_0^{N-k} (u'(r_0))^k+\int_{r_0}^r t^{N-1}(u'(t))^m v^p(t) dt\quad\mbox{ for all }0<r_0<r<R.
$$
Thus, if $u'(r)=0$ for some $0<r<R$, then $u'\equiv 0$ in $[0,r]$ and from the second equation of \eqref{main} we also get $v'\equiv 0$ on $[0, r]$, contradiction.

Integrating in the first equation of \eqref{main} we find
$$
r^{N-k}(u'(r))^k =\int_0^r t^{N-1}(u'(t))^m v^p(t) dt\leq r^{N-k} (u'(r))^k\int_0^r t^{k-1}(u'(t))^{m-k} v^p(t) dt,
$$
for all $0<r<R$. Hence
$$
1\leq \int_0^r t^{k-1}(u'(t))^{m-k} v^p(t) dt
\quad\mbox{ for all }0<r<R.
$$ Observe now that the integrand in the above estimate is a continuous function, so the right hand-side integral converges to zero as $r\to 0^+$, contradiction. Hence, $k>m$.

Let us rewrite the system \eqref{main} in the form
\begin{equation}\label{k1}
\left\{
\begin{aligned}
&\big[(u')^{k}\big]'(r)+\frac{N-k}{r}(u'(r))^{k}=r^{k-1}(u'(r))^{m}v^{p}(r) &&\quad\mbox{ for all } 0<r<R,\\
&\big[(v')^{k}\big]'(r)+\frac{N-k}{r}(v'(r))^{k}=r^{k-1} (u'(r))^{q} v^{s}(r)  &&\quad\mbox{ for all } 0<r<R,\\
&u'(0)=v'(0)=0, u(r)>0, v(r)>0 &&\quad\mbox{ for all } 0<r<R,
\end{aligned}
\right.
\end{equation}

We rearrange the system \eqref{k1} as
\begin{equation}\label{k2}
\left\{
\begin{aligned}
&\big[(u')^{k-m}\big]'(r)+\frac{L}{r}(u'(r))^{k-m}=\Big(1-\frac{m}{k}\Big) r^{k-1} v^{p}(r) &&\quad\mbox{ for all } 0<r<R,\\
&\big[(v')^{k}\big]'(r)+\frac{N-k}{r}(v'(r))^{k}=r^{k-1} (u'(r))^{q} v^{s}(r)  &&\quad\mbox{ for all } 0<r<R,\\
&u'(0)=v'(0)=0, u(r)>0, v(r)>0 &&\quad\mbox{ for all } 0<r<R,
\end{aligned}
\right.
\end{equation}
where
\begin{equation}\label{LL}
L=\frac{(N-k)(k-m)}{k}>0.
\end{equation}
From \eqref{k2} we have
\begin{equation}\label{k3}
\left\{
\begin{aligned}
&\Big[r^{L}(u')^{k-m}\Big]'(r)=\frac{k-m}{k} r^{k+L-1}v^{p}(r) &&\quad\mbox{ for all } 0<r<R,\\
&\Big[ r^{N-k}(v')^{k}\Big]'(r)=r^{N-1} (u'(r))^{q}  v^{s}(r)  &&\quad\mbox{ for all } 0<r<R.
\end{aligned}
\right.
\end{equation}

\medskip

Before we proceed with the proof of Theorem \ref{thm1} we need to establish two auxiliary results. The first lemma below provides basic estimates for solutions of \eqref{k2}. 

\begin{lemma}\label{firstestim}
Any non-constant positive radial solution $(u, v)$ of \eqref{k3}
in $B_R$ satisfies
\begin{equation}\label{k4}
\Big(N+\frac{km}{k-m}\Big) (u'(r))^{k-m}  < r^{k} v^{p}(r)
\quad\mbox{ for all }0<r<R,
\end{equation}

\begin{equation}\label{k5}
\frac{k(k -m)}{kN-(N-k) m}  r^{k-1} v^{p}(r) < \big[(u')^{k-m}\big]'(r)
<\Big(1-\frac{m}{k}\Big) r^{k-1} v^{p} 
\quad\mbox{ for all }0<r<R,
\end{equation}

\begin{equation}\label{k6}
(v'(r))^{k}  <\frac{1}{N}   r^{k} (u'(r))^{q}v^{s}(r)
\quad\mbox{ for all }0<r<R,\end{equation}
and
\begin{equation}\label{k7}
\frac{k}{N} v^{s}(r)(u'(r))^{q}r^{k-1}   \leq [(v')^{k}]'(r)  \leq r^{k-1} v^{s}(r)(u'(r))^{q} \quad\mbox{ for all }0<r<R.
\end{equation}

\end{lemma}

\begin{proof} We integrate the first equation of \eqref{k3} and using the fact that $v$ is strictly increasing on $(0, R)$ we find
\begin{equation}\label{estimates2}
\begin{aligned}
r^{L}(u'(r))^{k-m}&=\Big(1-\frac{m}{k}\Big) \int_0^r t^{k+L-1}v^{p}(t) dt\\ 
&<\Big(1-\frac{m}{k}\Big) v^{p}(r)  \int_0^r t^{k+L-1}dt\\ 
&=\frac{k -m}{\beta(k+L)}   r^{k+L} v^{p}(r)
\quad\mbox{ for all }0<r<R.
\end{aligned}
\end{equation}
This implies,
$$
(u'(r))^{k-m}  <\frac{k -m}{kN -(N-k) m}   r^{k} v^{p}(r)
\quad\mbox{ for all }0<r<R,
$$
which proves \eqref{k4}. We next use the above estimate in the first equation of \eqref{k2} to deduce
$$
\begin{aligned}
\Big(1-\frac{m}{k}\Big) r^{k-1} v^{p}(r) & = \big[(u')^{k-m}\big]'+\frac{L}{r}(u'(r))^{k-m}\\
& < \big[(u')^{k-m}\big]'+\frac{L(k -m)}{kN-(N-k) m}  r^{k-1} v^{p}(r),\\
\end{aligned}
$$
and this yields
$$
\big[(u')^{k-m}\big]'(r)>\frac{k(k -m)}{kN-(N-k) m}  r^{k-1} v^{p}(r) \quad\mbox{ for all }0<r<R.
$$
This is exactly the first half of the estimate in \eqref{k5}. The second half of \eqref{k5} follows immediately from \eqref{k5} since $u'>0$. 
Let us note that from \eqref{k5} we have
that $(u')^{k-m}$ is positive and strictly increasing so $u'$ is also positive and strictly increasing. Using this fact, the estimates \eqref{k6}-\eqref{k7} are derived in a similar fashion.

\end{proof}

Denote
\begin{equation}\label{pss}
\Psi(r)=(u'(r))^{k-m}\quad\mbox{ for }0\leq r<R.
\end{equation}
The next result provides a refinement of the estimates in Lemma \ref{firstestim}.
\begin{lemma}\label{secondestim}
Let $(u, v)$ be a solution of \eqref{k3}. Then, there exist constants $c,c_1,c_2>0$ and $\rho\in (0, R)$
such that 
\begin{equation}\label{k01}
cv^p(r)\leq \Psi^\sigma(r)  \quad\mbox{ for all } \rho\leq r<R
\end{equation}
and
\begin{equation}\label{k02}
c_1\leq \Psi'(r)\Psi^{-\sigma}(r)\leq c_2 \quad\mbox{ for all } \rho\leq r<R,
\end{equation}
where
\begin{equation}\label{ks}
\sigma=\frac{p}{k-m} \cdot \frac{q+(k-m)(1+k)}{(p+1)k+p-s}>0.
\end{equation}
\end{lemma}

\begin{proof} From \eqref{k5} and \eqref{k7} we find two positive constants $C>c>0$ such that
\begin{equation}\label{k8}
c r^{k-1} v^{p}(r) < \Psi'(r)
<C r^{k-1}  v^{p}(r)
\quad\mbox{ for all }0\leq r<R,
\end{equation}
and
\begin{equation}\label{k9}
c r^{k-1}  v^{s}(r) \Psi^{\frac{q}{k-m}}(r)  \leq [(v')^{k}]'(r)  \leq C r^{k-1}  v^{s}(r) \Psi^{\frac{q}{k-m}}(r)\quad\mbox{ for all }0\leq r<R.
\end{equation}
We multiply \eqref{k8}-\eqref{k9} we find
$$
\Psi^{\frac{q}{k-m}}(r)\Psi'(r)\leq c v^{p-s}(r)  [(v')^{k}]'(r) \quad\mbox{ for all }0\leq r<R.
$$
Next, we integrate over $[0,r]$. Since  $\Psi(0)=0$ and $v$ is increasing we find
$$
\begin{aligned}
\Psi^{\frac{q+k-m}{k-m}}(r) &\leq C\int_0^r v^{p-s}(t)  [(v')^{k}]'(t) dt\\
& \leq C v^{p-s}(r) \int_0^r   [(v')^{k}]'(t) dt=C v^{p-s}(r) (v'(r))^{k} 
\quad\mbox{ for all }0\leq r<R.
\end{aligned}
$$
Hence
\begin{equation}\label{k10}
\Psi^{\frac{q+k-m}{k(k-m)}}(r)\leq C v^{\frac{p-s}{k}}(r)  v'(r)   
\quad\mbox{ for all }0\leq r<R.
\end{equation}

Take $\rho\in (0, R)$. From \eqref{k8} we find 
\begin{equation}\label{k11}
c_1 v^{p}(r) < \Psi'(r)
<c_2  v^{p}(r)
\quad\mbox{ for all }\rho\leq r<R.
\end{equation}
We multiply \eqref{k10} by $\Psi'(r)$. Using $\Psi'(r)<c_2  v^{p}(r)$ for all $\rho\leq r<R$, we derive
$$
\Psi^{\frac{q+k-m}{k(k-m)}}(r)\Psi'(r)\leq C v^{p+\frac{p-s}{k}}(r)  v'(r)   \quad\mbox{ for all }\rho\leq r<R.
$$
Integrate now the above inequality over $[\rho,r]$. Since $v$ is continuous and positive on $[\rho, R)$, by taking a larger constant $C>0$ such that
$$
\Psi^{\frac{q+(k-m)(1+k)}{k-m}}(r)\leq C v^{(p+1)k+p-s}(r)    \quad\mbox{ for all }\rho\leq r<R.
$$
Using this last estimate together with \eqref{k11} we write
$$
\Psi^{\frac{q+(k-m)(1+k)}{k-m}}(r)\leq C \big(v^p(r)\big)^{\frac{(p+1)k+p-s}{p}}\leq C \Big(\Psi'(r)\Big)^{\frac{(p+1)k+p-s}{p}}    \quad\mbox{ for all }\rho\leq r<R,
$$
which yields
\begin{equation}\label{k12}
\Psi^{\sigma}(r)\leq C \Psi'(r) \quad\mbox{ for all }\rho\leq r<R.
\end{equation}

On the other hand, from \eqref{k9} we have
$$
[(v')^{k}]'(r)  \leq C v^{s}(r) \Psi^{\frac{q}{k-m}}(r)\quad\mbox{ for all }\rho\leq r<R.
$$
Multiply by $v'$ and integrate over $[\rho, r]$ in the above inequality. We find 
$$
\begin{aligned}
(v')^{k+1}(r)- (v')^{k+1}(r_1)  &\leq C \Psi^{\frac{q}{k-m}}(r) \int_{\rho}^r v^{s}(t)v'(t) dt \\
&\leq C \Psi^{\frac{q}{k-m}}(r)  v^{s+1}(r) \quad\mbox{ for all }\rho\leq r<R.
\end{aligned}
$$
Again by continuity arguments and by enlarging the value of $C>0$, one has
$$
(v')^{k+1}(r) \leq C \Psi^{\frac{q}{k-m}}(r)  v^{s+1}(r) \quad\mbox{ for all }\rho\leq r<R,
$$
that is,
$$
v^{-\frac{s+1}{k+1}}(r) v'(r) \leq C \Psi^{\frac{q}{(k+1)(k-m)}}(r)   \quad\mbox{ for all }\rho\leq r<R.
$$
Multiply the above inequality by $\Psi'(r)$. Using \eqref{k11} we find 
$$
v^{p-\frac{s+1}{k+1}}(r) v'(r) \leq C \Psi^{\frac{q}{(k+1)(k-m)}}(r)  \Psi'(r) \quad\mbox{ for all }\rho\leq r<R.
$$
A new integration over $[\rho, r]$ a continuity argument and by taking a larger constant $C>0$ one deduces
$$
v^{(p+1)k+p-s}(r)\leq C \Psi^{\frac{q+(k+1)(k-m)}{k-m}}(r) \quad\mbox{ for all }\rho\leq r<R,
$$
which yields \eqref{k01}.

Using again \eqref{k11} from which we have $v^p(r)\geq c\Psi'(r)$, we find
$$
c \Big(\Psi'(r)\Big)^{\frac{(p+1)k+p-s}{p}} \leq 
v^{(p+1)k+p-s}(r) \leq C \Psi^{\frac{q+(k+1)(k-m)}{k-m}}(r) \quad\mbox{ for all }\rho\leq r<R,
$$
that is, 
\begin{equation}\label{k13}
c  \Psi'(r)\leq  \Psi^{\sigma}(r)   \quad\mbox{ for all }\rho\leq r<R,
\end{equation}
where $\sigma>0$ is given by \eqref{ks}.
We next combine \eqref{k12} and \eqref{k13} to deduce
\begin{equation}\label{k14}
c_1\leq \Psi'(r)\Psi^{-\sigma}(r)\leq c_2 \quad\mbox{ for all } \rho\leq r<R,
\end{equation}
for some $c_2>c_1>0$. This is exactly inequality \eqref{k02}.

\end{proof}

\noindent{\bf Proof of Theorem \ref{thm1} completed.}

(i) Assume that $(u,v)$ is a solution of \eqref{main} with $u(R^{-})= \infty$ and $v(R^{-})< \infty$. Since $r^{N-k}(u')^k$ is increasing
(from the first equation of \eqref{main})  we deduce that
$u'(R^{-})= \infty$. Also, from \eq{k5} we find

\begin{equation*}
\big[(u')^{k-m}\big]'(r)
<C r^{k-1} v^{p} 
\quad\mbox{ for all }0<r<R.
\end{equation*}
for some positive constants $C>0$. Integrating over $[0, R]$  we obtain
$$
(u'(R^-))^{k-m}<C\int_0^R t^{k-1}v^p(t)dt<\infty,
$$
which contradicts $u'(R^{-})= \infty$.

(iii)-(iv) Let $(u,v)$ be a solution of \eqref{main} with $v(R^-)=\infty$ and let $\Psi$ be defined by \eqref{pss}. From \eqref{k01} it follows that
$\Psi(R^-)=\infty$.  We integrate over $[r, R]$ in \eqref{k02} to deduce $\sigma>1$ and
$$
c_1(R-r)^{-\frac{1}{(k-m)(\sigma-1)}} \leq u'(r)\leq c_2 (R-r)^{-\frac{1}{(k-m)(\sigma-1)}}\quad\mbox{ for all } \rho\leq r<R.
$$
This shows that
$$
\begin{aligned}
u(R^-)=u(\rho)+\int_{\rho}^R u'(r)dr <\infty & \Longleftrightarrow \int_{r_3}^R u'(r)dr<\infty \\
&\Longleftrightarrow \int_{\rho}^R (R-r)^{-\frac{1}{(k-m)(\sigma-1)}} dr<\infty \\
&\Longleftrightarrow \int_{0}^1 t^{-\frac{1}{(k-m)(\sigma-1)}} dt<\infty \\
&\Longleftrightarrow  \sigma>\frac{k-m+1}{k-m},
\end{aligned}
$$
and
$$
u(R^-)=\infty \Longleftrightarrow \sigma\leq \frac{k-m+1}{k-m}.
$$
Conversely, assume now that $\sigma>1$. The existence of a local non-constant positive solution to \eqref{k3} in a
small ball $B_\rho$ follows from standard fixed point arguments; see e.g., \cite[Proposition A1]{FV2017}
and \cite[Proposition 9]{BFP2015}. More precisely, the mapping
\begin{equation}\label{t1}
{\mathcal T}:C^1[0,\rho]\times C^1[0,\rho]\to C^1[0,\rho]\times C^1[0,\rho],
\end{equation}
defined by
\begin{equation}\label{t2}
{\mathcal T}[u,v](r)=\left[\begin{array}{c}{\mathcal
T_1}[u,v](r)\\{\mathcal T_2}[u,v](r)\end{array}\right],
\end{equation}
where
\begin{equation}\label{t3}
\left\{
\begin{aligned}
&&{\mathcal T_1}[u,v](r)&=a+\int_0^r\left(\frac{k-m}{k}t^{-L}\int_0^t \tau^{k+L-1}  v^p(\tau)d\tau \right)^{1/(k-m)}dt, \\
&&{\mathcal T_2}[u,v](r)&=b+\int_0^r\left(t^{k-N}\int_0^t \tau^{k-1}
v^s |u'(\tau)|^q d\tau\right)^{1/k}dt,
\end{aligned}
\right.
\end{equation}
and $a,b>0$, has a fixed point in $C^1[0,\rho]\times C^1[0,\rho]$ provided $\rho>0$ is small enough. 
Further, the scaling $(u_\lambda,v_\lambda)$ defined as
$$
u_\lambda(x)=\lambda^{1+\frac{p(2k-1)+k(k-s)}{\delta}}u\Big(\frac{x}{\lambda}\Big)\,,\quad v_\lambda(x)=\lambda^{\frac{(2k-1)(k-m)+kq}{\delta}}v\Big(\frac{x}{\lambda}\Big),
$$
provides a non-constant positive radially symmetric solution of \eqref{k3} in the
ball $B_{\lambda\rho}$. This shows that in any ball of positive radius there are non-constant positive
radially symmetric solution of \eqref{hess}.

Let now $(u,v)$ be a positive non-constant solution of \eqref{main} in a maximum interval $[0, R_{max})$. We claim that if $\sigma>1$ then $v(R^-_{max})=\infty$. Using the estimate \eqref{k02} in Lemma \ref{secondestim} on obtains after integrating over $[\rho, r]$ that
$$
c_1(r-\rho)\leq \frac{1}{\sigma-1}\Big(\Psi^{1-\sigma}(\rho)-\Psi^{1-\sigma}(r)\Big)\quad\mbox{ for all }\rho<r<R_{max}.
$$
Hence, by letting $r\to R_{max}^-$ one gets
$$
c_1R_{max}\leq C_1\rho+\frac{1}{\sigma-1}\Psi^{1-\sigma}(\rho)<\infty.
$$
This implies $R_{max}<\infty$ and then, using part (i) above, one deduces that $v(R^-_{max})=\infty$. 

In conclusion we found
\begin{itemize}
\item There are positive solutions $(u,v)$ of \eqref{main} with $u(R^-)<\infty$ and $v(R^-)<\infty$ if and only if $\sigma>\frac{k-m+1}{k-m}$;
\item There are positive solutions $(u,v)$ of \eqref{main} with $u(R^-)=v(R^-)<\infty$ if and only if $1<\sigma\leq \frac{k-m+1}{k-m}$;
\item All positive solutions $(u,v)$ of \eqref{main} are bounded if and only if $\sigma<1$.
\end{itemize}
Note that the case $\sigma=1$ is excluded from our analysis by the assumption \eqref{deltat}.
Now, the above conditions in terms of $\sigma$ are equivalent to (ii)-(iv) in the statement of Theorem \ref{thm1}. \qed

\section{Proof of Theorem \ref{thm3}}

Our approach to the study of the behaviour of solutions to \eqref{main} at infinity relies on some properties for three component dynamical systems obtained in Hirsch \cite{H1985}. For the reader's convenience, we shall briefly recall them below.

\subsection{Some results for cooperative dynamical systems}

Let ${\bf a}=(a_1,a_2,a_3), {\bf b}=(b_1,b_2,b_3)\in \R^3$. We say that $ {\bf a}\leq {\bf b}$ (resp. ${\bf a}<{\bf b}$) if
$a_i\leq b_i$ (resp. $a_i<b_i$) for all $1\leq i\leq 3$. 
We also define the closed interval  $[{\bf a},{\bf b}]=\{{\bf u}\in \R^3:{\bf a}\leq {\bf u}\leq {\bf b}\}$ and the open interval
$ [[{\bf a},{\bf b}]]=\{{\bf u} \in \R^3:{\bf a}< {\bf u}<{\bf b}\}$ with endpoints at ${\bf a}$ and ${\bf b}$.

A set $X\subset \R^3$ is said to be $p$-convex if the segment line
joining any two points in $X$ lies entirely in $X$.    Throughout this section $X\subset \R^3$ is assumed to be an open $p$-convex set.

Let ${\bf g}=(g_1,g_2,g_3):X\to \R^3$ be a $C^1$ cooperative vector field in the sense that 
$$
\frac{\partial g_i}{\partial x_j}\geq 0 \quad\mbox{ in }X\quad \mbox{  for any }\; i,j=1,2,3, \;\; i\neq j.
$$
For any $P\in \R^3$ we denote by $\Phi(t,P)$ the maximally defined
solution of the differential equation
\begin{equation}\label{flow}
\frac{d\zeta}{dt}={\bf g}(\zeta)
\end{equation}
subject to the initial condition $\zeta(0)=P$. The collection of maps$\{\Phi(t,\cdot)\}$ is called the flow of
the differential equation \eqref{flow}.

It is well know the following comparison property  for cooperative systems.

\begin{theorem}\label{comparis}{\rm (See \cite{H1985})}
Suppose ${\bf g}:X\to \R^3$ is a $C^1$ cooperative vector field and let $\zeta, \xi:[0,a]\to \R$, $a>0$, be two
solutions of \eqref{flow} such that
$$
\zeta(0)<\xi(0)\quad (\mbox{ resp.} \zeta(0)\leq \xi(0)\;).
$$
Then
$$
\zeta(t)<\xi(t)\quad (\mbox { resp.}  \zeta(t)\leq \xi(t)\;) \quad \mbox{ for all } t\in [0,a].
$$
\end{theorem}

For any point $P\in X$ we denote by $\omega(P)$ the $\omega$-limit set of $P$, that is, the set of all points $Q\in \R^3$ so that
there exists $\{t_j\}$, $t_j\to \infty$ (as $j\to \infty$) such that $\Phi(t_j, P)\to Q$ (as $j\to \infty$). Let also $E$ be the set of all equilibrium points of \eqref{flow}, that is, solutions  of ${\bf g}(\zeta)=0$.

Hirsch \cite{H1985} and then Hirsch and Smith \cite{HS2005} obtained that in any three component cooperative system  the omega limit sets preserve the partial order between the elements of $X$ or approach the equilibrium set $E$. This is summarised in the result below.

\begin{theorem}\label{dich}{\rm (Limit Set Dichotomy, see \cite[Theorem 3.8]{H1985}, \cite[Theorem 1.16]{HS2005})}

Suppose ${\bf g}:X\to \R^3$ is a $C^1$ cooperative vector field and let $P,Q\in X$, $P<Q$.
Then the following alternative holds:
\begin{enumerate}
\item[(i)] either $\omega(P)<\omega(Q)$;
\item[(ii)] or $\omega(P)=\omega(Q)\subset E$.
\end{enumerate}

\end{theorem}

A $C^1$-cooperative vector field ${\bf g}:X\to \R^3$ is said to be irreducible if at any point $P\in X$ its gradient $\nabla g(P)$ is
an irreducible matrix.
Hirsch \cite{H1985} showed that compact omega limit sets of cooperative and irreducible vector fields have a particular property in the
sense that they approach the equilibrium set for almost all points in $X$. 

\begin{theorem}\label{lebesgue}{\rm (See \cite[Theorem 4.1]{H1985})}

Suppose ${\bf g}:X\to \R^3$ is a $C^1$ cooperative and
irreducible vector field and that for all $P\in X$ the $\omega$-limit set
$\omega(P)$ is compact. Then, there exists $\Sigma\subset X$ with
zero Lebesgue measure such that
$$
\omega(P)\subset E\quad\mbox{ for all }\quad P\in X\setminus \Sigma,
$$
where $E$ denotes the set of equilibrium point of \eqref{flow}.
\end{theorem}

\subsection{Proof of Theorem \ref{thm3}}

Let $(u,v)$ be a non-constant global positive solution of
\eqref{main}. We introduce the change of variables
\begin{equation}\label{xyzw}
X(t)= \frac{ru'(r)}{u(r)},\;\;Y(t)= \frac{rv'(r)}{v(r)},\;\;Z(t)=\frac{r^kv^{p}(r)}{(u'(r))^{k-m}} \mbox{ and }W(t)=\frac{rv^{s}(r)(u'(r))^{q}}{(v'(r))^{k}}.
\end{equation}
where $t= \ln(r)\in \R$.
Thus, a direct calculation shows that $(X,Y,Z,W)$ satisfies the system
\begin{equation}\label{th2a}
\left\{
\begin{aligned}
&X_{t}= X\Big(\frac{2k-N}{k}-X+\frac{1}{k}Z\Big) \quad\mbox{ for all } t\in \R, \\
&Y_{t}= Y\Big(\frac{2k-N}{k}-Y+\frac{1}{k}W\Big) \quad\mbox{ for all } t\in \R, \\
&Z_{t}= Z\Big(\frac{kN-m(N-k)}{k}-\frac{k-m}{k}Z+pY\Big) \quad\mbox{ for all } t\in \R, \\
&W_{t}= W\Big(\frac{kN-q(N-k)}{k}+s Y+\frac{q}{k}Z-W\Big) \quad\mbox{ for all } t\in \R.
\end{aligned}
\right.
\end{equation}

Using L'Hopital's rule one has 
\begin{equation}\label{limit}
\lim_{t\rightarrow \infty}X(t)= \lim_{r\rightarrow \infty}\frac{ru'(r)}{u(r)}=\lim_{r\rightarrow \infty}\Big(1+\frac{ru''(r)}{u'(r)}\Big)=\lim_{t\rightarrow \infty}\Big(\frac{1}{k}Z(t)+\frac{k-N}{k}\Big),
\end{equation}
provided the limit $\lim_{t\rightarrow \infty}Z(t)$ exists. Thus, it is enough to study the system consisting of
the last three equations of \eq{th2a} which we arrange in the form
\begin{equation}\label{th2b}
\zeta_{t}= g(\zeta) \quad\mbox{ in } \R,
\end{equation}
where
\begin{equation}\label{zet}
\zeta(t)=\left[\begin{array}{c}Y(t)\\Z(t)\\W(t)\end{array}\right] \quad\mbox{ and } \quad g(\zeta)=  \left[\begin{array}{c} Y\Big(\frac{2k-N}{k}-Y+\frac{1}{k}W\Big)\\Z\Big(\frac{kN-m(N-k)}{k}-\frac{k-m}{p-1}Z+pY\Big) \\  W\Big(\frac{kN-q(N-k)}{p-1}+s Y+\frac{q}{k}Z-W\Big)  \end{array}\right].
\end{equation}
Among all the equilibrium points of \eq{th2b}-\eqref{zet}, only one has all components strictly positive namely
\begin{equation}\label{eqinf}
\zeta_\infty=\left[\begin{array}{c}Y_\infty\\Z_\infty\\W_\infty\end{array}\right],
\end{equation}
where
\begin{equation}\label{th2c}
\left\{
\begin{aligned}
&\frac{2k-N}{k}-Y_\infty+\frac{1}{k}W_\infty=0,  \\
&\frac{kN-m(N-k)}{k}-\frac{k-m}{k}Z_\infty+pY_\infty=0, \\
&\frac{kN-q(N-k)}{k}+s Y_\infty+\frac{q}{k}Z_\infty-W_\infty=0.
\end{aligned}
\right.
\end{equation}
Solving \eqref{th2c} we find
\begin{equation}\label{th2d}
\left\{
\begin{aligned}
&Y_\infty=\frac{kq+2k(k-m)}{\delta}, \\
&Z_\infty=\frac{kp}{k-m }Y_\infty+N+\frac{km}{k-m}, \\
&W_\infty= kY_\infty+N-2k.
\end{aligned}
\right.
\end{equation}

\begin{lemma}\label{l2b}
The equilibrium point $\zeta_\infty$ is asymptotically stable.
\end{lemma}
\begin{proof}
Using the equalities in \eq{th2c} we compute the linearized matrix of \eq{th2b} at $\zeta_\infty$  as follows:
$$
M_\infty=\left[\begin{array}{ccc}-Y_\infty&0&\frac{1}{k}Y_\infty\\ pZ_\infty&-\frac{k-m}{k}Z_\infty&0\\ s W_\infty&\frac{q}{k}W_\infty&-W_\infty\end{array}\right].
$$
Thus, the characteristic polynomial of $M_\infty$ is
$$
P(\lambda)=\det(\lambda I-M)=\lambda^{3}+a\lambda^{2}+b \lambda+c,
$$
where
\begin{subequations}
\begin{align}
a&=Y_\infty+\frac{k-m}{k}Z_\infty+W_\infty, \label{ee1}\\
b&=\frac{k-m}{k}Y_\infty Z_\infty+\frac{k-s}{k}Y_\infty W_\infty+\frac{k-m}{k}Z_\infty W_\infty, \label{ee2}\\
c&= \frac{\delta}{k^2}Y_\infty Z_\infty W_\infty.\label{ee3}
\end{align}
\end{subequations}

We divide our argument into two steps. 
\medskip

\noindent{\it Step 1:}  $ab>9c$. Note that from $\delta>0$, $k>m$ and \eqref{deltat} we have $k>s$. Hence, using $Y_\infty$, $Z_\infty$, $W_\infty>0$ we estimate \eqref{ee1} as follows
$$
a\geq \frac{k-s}{k}Y_\infty+\frac{k-m}{k}Z_\infty+\frac{k-s}{k}W_\infty.
$$
Thus, by AM-GM inequality we find
$$
a\geq 3\frac{(k-m)^{\frac{1}{3}}(k-s)^{\frac{2}{3}}}{k} (Y_\infty Z_\infty W_\infty)^{\frac{1}{3}}.
$$
In a similar fashion, from \eqref{ee2} and AM-GM inequality we estimate 
$$
b\geq
3\frac{(k-m)^{\frac{2}{3}}(k-s)^{\frac{1}{3}}}{p-1} (Y_\infty Z_\infty W_\infty)^{\frac{2}{3}}.
$$
We now multiply the above inequalities to deduce
$$
ab\geq 9 \frac{(k-m)(k-s)}{k^2} (Y_\infty Z_\infty W_\infty)>9c.
$$

\medskip

\noindent{\it Step 2:} all three roots $\lambda_1$, $\lambda_2$ and $\lambda_3$ of the characteristic
polynomial $P(\lambda)$ of $M_\infty$ have negative real part. Indeed, if $\lambda_i \in \R$, for all
$i=1$, $2$, $3$ then, since $a,b,c>0$ it follows $P(\lambda)>0$ for all $\lambda \geq 0$ so that $\lambda_i< 0$ for all
$i=1$, $2$, $3$. If $P$ has exactly one real root, say $\lambda_1 \in \R$, then ${\rm Re}(\lambda_2)={\rm Re}(\lambda_3)$. Using $P(-a)=-ab+c<0$, it follows that $\lambda_1> -a$. Since $\lambda_1+\lambda_2+\lambda_3= -a$ we easily deduce that ${\rm Re}(\lambda_2)={\rm Re}(\lambda_3)<0$. This proves that $\zeta_\infty$ is asymptotically stable.

\end{proof}

\medskip

\begin{lemma}\label{l2c}
The following estimates hold for all $t\in \R$:
\begin{equation}\label{infy}
0<Y(t)< Y_\infty\,,\qquad N+\frac{km}{k-m}<Z(t)<Z_\infty\,,\qquad N<W(t)<W_\infty.
\end{equation}
\end{lemma}
\begin{proof}
We proceed into four steps.
\medskip

\noindent{\it Step 1: \it Preliminary estimates:} 
\begin{equation}\label{stepp1}
Z(t)> N+\frac{km}{k-m}\,, \quad W(t)> N\quad\mbox{ for all } t\in \R\quad\mbox{
and }\lim_{t\rightarrow -\infty}Y(t)=0.
\end{equation}

\medskip

The lower bounds for $Z$ and $W$ follow from \eqref{k4} and \eqref{k5} in Lemma \ref{firstestim}.
Since $v'(0)=0$ and $v(0)>0$ we have $\lim_{t\rightarrow -\infty}Y(t)= \lim_{r\rightarrow 0}\frac{rv'(r)}{v(r)}= 0$.

\bigskip

\noindent{\it Step 2: There exists $T\in \R$ such that $Z(t)< Z_\infty$ for all $t\in (-\infty, T]$.}
\medskip

It is enough to show that
\begin{equation}\label{zlim}
\lim_{t\to-\infty} Z(t)=N+\frac{km}{k-m}.
\end{equation}
To this aim, we shall use the Generalized Mean Value
Theorem\footnote{Generalized Mean Value Theorem (or Cauchy's Theorem) states that if
$f$,$g: [a, b]\rightarrow \R$ are differentiable functions on $(a, b)$ and continuous on $[a, b]$,
then there exists $c\in (a, b)$ such that $\frac{f(b)-f(a)}{g(b)-g(a)}= \frac{f'(c)}{g'(c)}$.}
\cite[Theorem 5.9, page 107]{R1976}.

Let $t\in (-\infty, 0)$ and $r= e^t \in (0, 1)$. 
From the first equation of \eqref{k2} we have
$$
\frac{d}{dr}\Big[(u')^{k-m}(r)\Big]=\frac{k-m}{k}\Big[ r^{k-1} v^{p}(r) -\frac{N-k}{r}(u'(r))^{k-m}\Big].
$$
Using this fact and the Generalized Mean Value
Theorem, there exists $c\in (0, r)$ such that
\begin{equation*}
\begin{aligned}
Z(t)&= \frac{rv^{m}(r)}{(u')^{k-m}(r)}=\frac{\displaystyle\frac{d}{dr}\Big[rv^{m}(r)\Big](c)}{\displaystyle \frac{d}{dr}\Big[(u')^{k-m}(r)\Big](c)}\\
&= \frac{k}{k-m}\cdot \frac{c^{k-1} v^p(c) +pc^k v^{p-1}(c) v'(c)}{c^{k-1} v^{p}(c) -\frac{N-k}{r}(u'(c))^{k-m}}.
\end{aligned}
\end{equation*}
Hence
\begin{equation}\label{th2f}
Z(t)= \frac{k}{k-m}\cdot\frac{Z({\rm ln}c)}{Z({\rm ln}c)-(N-k)} \Big(k+pY({\rm ln}c)\Big).
\end{equation}
Recall that from Step 1 above we have 
\begin{equation}\label{z1}
Z(t)>N+\frac{km}{k-m}\quad\mbox{ for all }\quad t\in\R.
\end{equation}
Thus,
$$
\frac{Z}{Z-(N-k)}<\frac{k-m}{k^2}\cdot \Big( N+\frac{km}{k-m}\Big).
$$
It now follows from \eq{th2f} that
\begin{equation}\label{z2}
\lim{\rm sup}_{t\rightarrow -\infty}Z(t)\leq \frac{1}{k}\Big( N+\frac{km}{k-m}\Big) \big(k+p\lim_{t\rightarrow -\infty}Y(t)\big)= N+\frac{km}{k-m}.
\end{equation}
From \eqref{z1}, the opposite inequality is also true. Hence \eqref{zlim} holds which implies that there exists $T\in \R$ such that $Z(t)< Z_\infty$ for all $t\leq T$.

\bigskip

\noindent{\it Step 3:
There exists a sequence $t_j \rightarrow -\infty$ such that
\begin{equation}\label{tj}
Y(t_j)< Y_\infty\,, \quad Z(t_j)< Z_\infty \quad \mbox{ and } \quad W(t_j)< W_\infty \quad\mbox{ for all } j\geq 1.
\end{equation}
}
Suppose the above inequalities do not hold. By taking $T\in \R$ sufficiently close to $-\infty$ and in light of the estimates already obtained at Step 1 and 2 above, we may assume
\begin{equation}\label{asump}
Y(t)< Y_\infty\,, \quad Z(t)< Z_\infty \quad \mbox{ and }\quad W(t)\geq W_\infty \quad\mbox{ for all } t\leq T.
\end{equation}
Using these estimates in the differential equation satisfied by $W$ in \eq{th2a} we deduce $W_t< 0$ on $(-\infty, T]$. Thus, $W$ is decreasing
in a neighbourhood of $-\infty$ and thus exists
$$
\ell:= \lim_{t\rightarrow -\infty} W(t)= \lim_{r\rightarrow 0}\frac{rv^{s}(r)(u'(r))^{q}}{(v'(r))^k}.
$$
From \eqref{asump} one has
\begin{equation}\label{ell1}
\ell\geq W_\infty.
\end{equation}
On the other hand, using \eqref{k1} one has
$$
\begin{aligned}
\frac{d}{dr}\big[(v')^{k}\big]'(r)=&\; r^{k-1} (u'(r))^{q} v^{s}(r) -\frac{N-k}{r}(v'(r))^{k}\\
\frac{d}{dr}\big[r^k v^s(u')^{q}\big]'(r)=&\; kr^{k-1}v^s(r)(u'(r))^q+sr^k v'(r)(u'(r))^q+\frac{q}{k}r^k v^{s+p}(r)(u'(r))^{q-k+m}\\
&-\frac{q(N-k)}{k}r^{k-1}(u'(r))^q v^s(r).
\end{aligned}
$$

Let $t\in (-\infty, T]$ and $r= e^{t}$. Applying the Generalized Mean Value Theorem as in the previous step
and using the above equalities we find $c\in (0, r)$ such that
\begin{equation}\label{th2g}
W(t)= \frac{W({\rm ln}c)}{W({\rm ln}c)-(N-k)}\left[k-\frac{q(N-k)}{k}+s Y({\rm ln}c)+\frac{q}{k}Z({\rm ln}c)\right].
\end{equation}
Recall that by \eqref{stepp1} we have $Z> N$ and $W> N$ so that right hand side of \eq{th2g} is positive.
Letting $t\rightarrow -\infty$ (that is, $c\rightarrow 0$) in \eqref{th2g} and using
$\lim_{t\rightarrow -\infty}Z(t)< Z_\infty$ and $\lim_{t\rightarrow -\infty}Y(t)= 0$  we obtain
\begin{equation}\label{th2h}
\ell= \lim_{t\rightarrow -\infty}W(t)\leq \frac{\ell}{\ell-(N-k)}\left[k-\frac{q(N-k)}{k}+\frac{q}{k}Z_\infty\right].
\end{equation}
This yields
$$
\ell\leq N-\frac{q(N-k)}{k}+\frac{q}{k}Z_\infty.
$$
Comparing this inequality with the last equation of \eqref{th2c} we find $\ell<W_\infty$, which contradicts \eqref{ell1}. This proves that \eqref{tj} holds.

\bigskip

\noindent{\it Step 4:  Conclusion of the proof. }
\medskip

We can compare now the solution $\zeta=(Y,Z,W)$ and the equilibrium point $\zeta_\infty=(Y_\infty, Z_\infty, W_\infty)$ 
on each of the intervals $[t_j,\infty)$. Thanks to Theorem \ref{comparis} we deduce
\begin{equation}\label{et}
Y(t)< Y_\infty\,, \quad Z(t)< Z_\infty \quad \mbox{ and } \quad W(t)< W_\infty \quad\mbox{ for all } t\geq t_j.
\end{equation}
Since $t_j\to -\infty$ it follows that the estimates in \eqref{et} hold for all $t\in \R$ and this together with
Step 1 proves \eqref{infy}.
\end{proof}

We are now able to complete the proof of Theorem \ref{thm3}.  Let $(u,v)$ be a non-constant global positive radial
solution of system \eqref{hess} and denote by $(X,Y,Z,W)$ the corresponding solution of
\eqref{th2a} as described in
\eqref{xyzw}.

Then $
\zeta(t)=\left[\begin{array}{c}Y(t)\\Z(t)\\W(t)\end{array}\right] $
satisfies \eqref{th2b}-\eqref{zet}. Thus, by  Lemma \ref{l2c}
we have
$$
\zeta_*:=\left[\begin{array}{c}0\\ \displaystyle N+\frac{km}{k-m}\\N\end{array}\right]<\zeta(0).
$$
Denote by $E\subset \R^3$ the set of equilibrium points associated with \eqref{th2b}-\eqref{zet}.
From the result in Theorem \ref{lebesgue} there exists a set $\Sigma\subset \R^3$ of Lebesgue measure zero such that
\begin{equation}\label{omegaa}
\omega(\widetilde \zeta)\subseteq E\quad\mbox{ for all }\quad \widetilde \zeta\in [\zeta_*,\zeta_\infty]\setminus\Sigma.
\end{equation}
For $\widetilde \zeta\in [\zeta_*,\zeta_\infty]\setminus\Sigma$ denote by
$$
\Phi(t,\widetilde \zeta)=\left[\begin{array}{c}\widetilde{Y}(t)\\\widetilde{Z}(t)\\\widetilde{W}(t)\end{array}\right]
$$
the flow of \eqref{th2b}  associated with the initial data $\widetilde \zeta\in \R^3$. Using Theorem \ref{comparis} and the fact that $\widetilde \zeta\geq \zeta_*$ we find
$$
\Phi(t,\widetilde \zeta)\geq \left[\begin{array}{c}0\\ \displaystyle N+\frac{km}{k-m}\\0\end{array}\right]\quad\mbox{ for all }t\geq 0.
$$
Hence, $\omega(\widetilde \zeta)$ is finite and consists of equilibrium points in $E$ whose second component 
is greater than or equal to $N+\frac{km}{k-m}$. It follows that
$$
\omega(\widetilde\zeta)\subseteq \{\zeta_1,\zeta_2,\zeta_3,\zeta_\infty\},
$$
where
$$
\zeta_1=\left[\begin{array}{c}0\\ \displaystyle
N+\frac{km}{k-m}\\0\end{array}\right]\;,\quad
\zeta_2=\left[\begin{array}{c}0\\ \displaystyle
N+\frac{km}{k-m}\\
\displaystyle N+\frac{qk}{k-m}\end{array}\right]\;,\quad
\zeta_3=\left[\begin{array}{c}\displaystyle \frac{2k-N}{k}\\
\displaystyle
N+\frac{mk+p(2k-N)}{k-m}\\0\end{array}\right]
$$
and $\zeta_\infty$ is given by \eqref{eqinf}. Note, that $\zeta_3$ has all components non-negative if and only of $2k\geq N$.

We claim that
\begin{equation}\label{eqfin}
\omega(\widetilde \zeta)=\{\zeta_\infty\}\quad\mbox{ for all }\quad
\widetilde \zeta\in [\zeta_*,\zeta_\infty]\setminus\Sigma.
\end{equation}

If $\zeta_\infty\in \omega(\widetilde \zeta)$ then, using Lemma \ref{l2b} we have that $\zeta_\infty$ is asymptotically stable, so that
$\omega(\widetilde \zeta)=\{\zeta_\infty\}$ and thus, \eqref{eqfin} follows.

 Assume in the following that $\zeta_\infty\not \in \omega(\widetilde \zeta)$ so
$\omega(\widetilde \zeta)\subseteq \{\zeta_1,\zeta_2,\zeta_3\}$.

If $\{\zeta_1,\zeta_2\}\subset \omega(\widetilde \zeta)$ or $\{\zeta_2,\zeta_3\}\subset \omega(\widetilde \zeta)$ then, along a subsequence $\widetilde W$ converges to 0 and 
$N+\frac{qk}{k-m}$.  By the Intermediate Value Theorem we deduce that for all $\tau\in (0,N+\frac{qk}{k-m})$
there exists a sequence $t_j\to \infty$ such that $\widetilde W(t_j)=\tau$ which contradicts the fact that
$\omega(\tilde \zeta)$ is finite. In a similar way, $\{\zeta_1,\zeta_3\}\subset \omega(\widetilde \zeta)$ we deduce that $2k>N$
and for any $0<\gamma<\frac{2k-N}{k}$ there exists a sequence $t_j\to \infty$ such that $\widetilde Y(t_j)=\gamma$
which again contradictions the fact that $\omega(\tilde \zeta)$ is finite.

The above arguments shows that  $\omega(\widetilde \zeta)$ reduces to a single element. We show in the following that this raises again a contradiction unless  $\omega(\widetilde \zeta)=\{\zeta_\infty\}$. 
Suppose for instance that $\omega(\widetilde \zeta)=\{\zeta_2\}$, that is 
$$
\lim_{t\to \infty}\widetilde Y(t)=0\,, \quad  \lim_{t\to \infty} \widetilde Z(t)= N+\frac{km}{k-m}\quad\mbox{and}\quad  \lim_{t\to \infty} \widetilde W(t)= N+\frac{qk}{k-m}.
$$
But then, for large $t>0$ we find
$$
\widetilde Y_t=\widetilde Y\Big(\frac{2k-N}{k}-\widetilde Y+\frac{1}{k}\widetilde W\Big)>0.
$$
This implies that $\widetilde Y$ is increasing in a neighbourhood of infinity. Hence,  for large $t>0$ we have $\widetilde Y(t)\leq \lim_{s\to
\infty}\widetilde Y(s)=0$, contradiction. In a similarl way, if
$\omega(\widetilde \zeta)=\{\zeta_1\}$ or if $\omega(\widetilde \zeta)=\{\zeta_3\}$
we raise a contradiction. This finally proves \eqref{eqfin}.

Take now $\zeta\in [[\zeta_*,\zeta_\infty]]\cap \Sigma$ and let
$\widetilde \zeta\in [\zeta_*,\zeta_\infty]\setminus \Sigma$ be such that
$\widetilde \zeta<\zeta$. By the Limit Set Dichotomy Theorem
\ref{dich}  the following alternative holds:
\begin{itemize}
\item either $\{\zeta_\infty\}=\omega(\widetilde \zeta)<\omega(\zeta)$;
\item or $\omega(\zeta)=\omega(\widetilde \zeta)=\{\zeta_\infty\}$.
\end{itemize}
The first alternative cannot hold. Indeed, by the comparison result in
Theorem \ref{comparis} it follows $\omega(\zeta)\leq \zeta_\infty$ which yields $\omega(\zeta)=\{ \zeta_\infty \}$ so,
$$
\omega(\zeta)=\{\zeta_\infty\}\quad\mbox{ for all } \zeta\in [[\zeta_*, \zeta_\infty]].
$$
In particular, 
$$
\omega(\zeta(0))=\{\zeta_\infty\},
$$
that is,
\begin{equation}\label{xax1}
\lim_{t\to \infty} Y(t)= Y_\infty\,,\qquad \lim_{t\to \infty} Z(t)= Z_\infty\,,\qquad \lim_{t\to \infty} W(t)= W_\infty.
\end{equation}
Also, from  \eq{limit} one has
\begin{equation}\label{xax2}
X_\infty:= \lim_{t\rightarrow \infty}X(t)= \frac{1}{k}Z_\infty+\frac{2k-N}{k}.
\end{equation}
A direct calculation shows that
$$
\frac{u^{\delta}(r)}{r^{\delta+k(k-s+2p)} }= \frac{1}{X^{\delta}(t) Y^{kp}(t)Z^{k-s}(t)W^{p}(t)} \quad\mbox{ for all } r> 0.
$$
Now, using \eqref{xax1}-\eqref{xax2} one has
\begin{equation}\label{th2i}
\lim_{r\rightarrow \infty} \frac{u(r)}{r^{1+\frac{k(k-s+2p)}{\delta}}}= A,
\end{equation}
where
$$
A= \frac{1}{X_\infty Y_\infty^\frac{kp}{\delta}Z_\infty^{\frac{k-s}{\delta}}W_\infty^{\frac{p}{\delta}}}\in (0, \infty).
$$
In a similar fashion we obtain
$$
\frac{v^{\delta}(r)}{r^{k(2k-2m+q)} }= \frac{1}{Y^{k(k-m)}(t)Z^{q}(t)W^{k-m}(t)} \quad\mbox{ for all } r> 0
$$
and again by \eqref{xax1}-\eqref{xax2} we derive
\begin{equation}\label{th2j}
\lim_{r\rightarrow \infty} \frac{v(r)}{r^{\frac{k(2k-2m+q)}{\delta}}}= B,
\end{equation}
where
$$
B= \frac{1}{Y_\infty^\frac{(k(k-m)}{\delta}Z_\infty^{\frac{q}{\delta}}W_\infty^{\frac{k-m}{\delta}}}\in (0, \infty).
$$

\section{Proof of Theorem \ref{thm4}}

\begin{equation}\label{main1}
\left\{
\begin{aligned}
&r^{1-N}\Big[ r^{N-k}|u'|^{k-1}u' \Big]'=|u'|^{m}v^{p},\; u(r)>0 &&\quad\mbox{ for all } r>0,\\
&r^{1-N}\Big[ r^{N-k}|v'|^{k-1}v' \Big]'=|u'|^{q} v^{s} ,\; v(r)>0 &&\quad\mbox{ for all } r>0,\\
&u'(0)=v'(0)=0, u(0)=a>0, v(0)=b>0.
\end{aligned}
\right.
\end{equation}
We use the change of variable
\begin{equation}\label{chvar}
r=t^{\theta},\quad \theta=-\frac{k}{N-2k}<0,
\end{equation}
and let 
$$
u(r)=U(t), \;\; v(r)=V(t).
$$
Then
$$
u'(r)=\frac{du}{dr}(r)=\frac{1}{\theta}t^{1-\theta}U_t(t)\quad\mbox{ and }\quad 
v'(r)=\frac{dv}{dr}(r)=\frac{1}{\theta}t^{1-\theta}V_t(t).
$$
Thus, any solution $(u, v)$ of \eqref{main1} satisfies
\begin{equation}\label{m1}
\left\{
\begin{aligned}
&\Big(|U_t|^{k-m-1}U_t \Big)_t=\frac{k}{k-m} |\theta|^{k-m+1} t^{(1-\theta)(m-1)-\theta(N-1)} V^p, \; U(t)>0 &&\quad\mbox{ for all } t>0,\\
&\Big(|V_t|^{k-1}V_t \Big)_t=|\theta|^{k-q+1} t^{(1-\theta)(q-1)-\theta(N-1)} |U_t|^q V^s, \; V(t)>0&&\quad\mbox{ for all } t>0,\\
& U_t, V_t<0, U(\infty)=u(0)>0, V(\infty)=v(0)>0, U_t(\infty)=V_t(\infty)=0.   
\end{aligned}
\right.
\end{equation}
Set 
$$
W(t)=|U_t|^{k-m-1}U_t \Longrightarrow |U_t|=|W|^{\frac{1}{k-m}},
$$
and now \eqref{m1} reads
\begin{equation}\label{m2}
\left\{
\begin{aligned}
&W_t=f(t)V^p, \; W(t)<0 &&\quad\mbox{ for all } t>0,\\
&\Big(|V_t|^{k-1}V_t \Big)_t=g(t) V^s |W|^{\frac{q}{k-m}}, \; V(t)>0&&\quad\mbox{ for all } t>0,\\
& V_t<0, W(\infty)=0, V(\infty)=v(0)>0, W(\infty)=V_t(\infty)=0.   
\end{aligned}
\right.
\end{equation}
where
$$
\left\{
\begin{aligned}
&f(t)=\frac{k}{k-m} |\theta|^{k-m+1} t^{(1-\theta)(m-1)-\theta(N-1)} &&\quad\mbox{ for all } t>0,\\
&g(t)=|\theta|^{k-q+1} t^{(1-\theta)(q-1)-\theta(N-1)}  V^s |W|^{\frac{q}{k-m}}&&\quad\mbox{ for all } t>0.
\end{aligned}
\right.
$$
Let now $(u,v)$ and $(\tilde u, \tilde v)$ be two solutions of \eqref{main1}. We want to show $u\equiv \tilde u$ and $v\equiv \tilde v$. Define $(U, V,W)$ and $(\tilde U, \tilde V, \tilde W)$ as above which satisfy problem \eqref{m1} and respectively
\begin{equation}\label{m3}
\left\{
\begin{aligned}
&\tilde W_t=f(t)\tilde V^p, \; \tilde W(t)<0 &&\quad\mbox{ for all } t>0,\\
&\Big(|\tilde V_t|^{k-1}\tilde V_t \Big)_t=g(t) \tilde V^s |\tilde W|^{\frac{q}{k-m}}, \; \tilde V(t)>0&&\quad\mbox{ for all } t>0,\\
& \tilde V_t<0, \tilde W(\infty)=0, \tilde V(\infty)=\tilde v(0)=b>0, \tilde W(\infty)=\tilde V_t(\infty)=0.   
\end{aligned}
\right.
\end{equation}
Let $\varepsilon >0$ be small and define
$$
u^\ve(r)=(1+\varepsilon)u(r),\quad v^\varepsilon (r)=(1+\varepsilon)^{\frac{k-m}{p}}v(r)\quad\mbox{ for all }r\geq 0.
$$
With the same change of variables \eqref{chvar} and 
$$
u^\ve(r)=U^\ve(t), \;\; v^\ve(r)=V^\ve(t),\;\;  W^\ve(t)=|U_t^\ve|^{k-m-1}U^\ve_t
$$ 
we obtain that $(W^\ve, V^\ve)$ satisfies
\begin{equation}\label{m4}
\left\{
\begin{aligned}
&W_t^\ve=f(t)(V^\ve)^p, \; W^\ve(t)<0 &&\quad\mbox{ for all } t>0,\\
&\Big(|V_t^\ve|^{k-1}V_t^\ve \Big)_t=(1+\ve)^{\frac{\delta}{p}} g(t) (V^\ve)^s |W^\ve|^{\frac{q}{k-m}}, \; V(t)>0&&\quad\mbox{ for all } t>0,\\
& V_t^\ve<0, W^\ve(\infty)=0, V^\ve(\infty)=(1+\ve)^{\frac{k-m}{p}}v(0)>0, W^\ve(\infty)=V^\ve_t(\infty)=0,  
\end{aligned}
\right.
\end{equation}
where $\delta$ is given by \eqref{deltat}.

Observe that $V^\ve(\infty)>\tilde V(\infty)$. Thus, from the first equation in \eqref{m3} and \eqref{m4} we find  that $W_t^\ve>\tilde W$ in a neighbourhood of infinity. Thus, the set
$$
M:=\{t>0: W_t^\ve>\tilde W \mbox{ on }(t, \infty)\}
$$
is non-empty. Set $t_0:=\inf M\geq 0$. 

\medskip

\noindent{\bf Claim: } $t_0=0$.

Assume by contradiction that $t_0>0$.  Then by continuity arguments one has
\begin{equation}\label{m5}
W_t^\ve>\tilde W\;\;\mbox{ on }\;\; (t_0, \infty) \quad\mbox{ and }W_t^\ve(t_0)=\tilde W_t(t_0).
\end{equation}
It follows from the first equation in \eqref{m3} and \eqref{m4} that
\begin{equation}\label{m6}
V^\ve>\tilde V\;\;\mbox{ on }\;\; (t_0, \infty) \quad\mbox{ and }V^\ve(t_0)=\tilde V(t_0).
\end{equation}
We next integrate over $[r, \infty]$ in \eqref{m5} to deduce
$$
|W^\ve(t)|=-W^\ve(t)>-\tilde W(t)=|\tilde W(t)|\quad\mbox{ for all }t>t_0.
$$
Using this fact in the second equation of \eqref{m3} and \eqref{m4} one has
$$
\Big(|V_t^\ve|^{k-1}V_t^\ve \Big)_t>\Big(|\tilde V_t|^{k-1}\tilde V_t \Big)_t \quad\mbox{ for all }t>t_0.
$$
An integration over $[t,\infty]$ in the above inequality yields
$$
-V_t^\ve(t)=|V_t^\ve(t)|>|\tilde V_t(t)|=-\tilde V(t)\quad\mbox{ for all }t>t_0.
$$
A further integration over $[t_0,\infty]$ in the above inequality together with the fact that $V^\ve(\infty)>\tilde V(\infty)$ implies 
$$
V^\ve(t_0)>\tilde V(t_0),   
$$
which contradicts the equality in \eqref{m6}. Hence, $t_0=0$ which proves our claim. This further yields
\begin{equation}\label{m7}
W^\ve_t>\tilde W_t\quad\mbox{ on }(0, \infty),
\end{equation} 
and integrating this inequality over $[t, \infty]$ one gets $|W^\ve(t)|>|\tilde W(t)|$ for all $t>0$. Hence
$$
|U^\ve(t)|>|\tilde U(t)|\quad\mbox{ for all }t>0.
$$
A further integration implies $U^\ve>\tilde U$ on $(0, \infty)$, that is, 
$$
U^\ve(t)=(1+\ve)u(r)>\tilde u(r)\quad\mbox{ for all }r>0.
$$
Passing to the limit with $\ve\to 0^+$ one has $u\geq \tilde u$ on $(0, \infty)$. Also, \eqref{m7} and the first equation in \eqref{m3} and \eqref{m4} imply $V^\ve>\tilde V$ on $(0, \infty)$ which yield $v\geq \tilde v$ on $(0, \infty)$. Hence, we have argued that $u\geq \tilde u$ and $v\geq \tilde v$ on $(0, \infty)$. Similarly $\tilde u\geq u$ and $\tilde v\geq v$ on $(0, \infty)$ which yields $u\equiv \tilde u$ and $v=\tilde v$.

\end{document}